\documentclass[hyperref]{article}
\usepackage[a4paper]{geometry}       
                
\usepackage{amsmath,amssymb,amsfonts}   

\usepackage{graphicx}
\usepackage{float}
\usepackage{subfigure}

\usepackage{mathtools}
\usepackage[thmmarks,amsmath]{ntheorem} 
\usepackage{bm}
\numberwithin{equation}{section}

\usepackage{theorem}
\newtheorem{theorem}{Theorem}[section]
\newtheorem{lemma}{Lemma}[section]
\newtheorem{remark}{Remark}[section]

\newtheorem{proof}{Proof}
\theoremstyle{nonumberplain}
\theoremheaderfont{\bfseries}
\theorembodyfont{\normalfont}

\begin{document}
	\title{\bf Local well-posedness for the Zakharov system in dimension $d=2, 3$} 
	\date{}
	\author{\sffamily Zijun Chen$^1$, Shengkun Wu$^{2}$\\
		{\sffamily\small $^1$ School of Mathematical Sciences, Monash University, Melbourne, VIC 3800, Australia}\\
		{\sffamily\small $^2$ College of Mathematics and Statistics, Chongqing University, Chongqing, 401331, PR }}
	\renewcommand{\thefootnote}{\fnsymbol{footnote}}
	\footnotetext[1]{Corresponding author: Zijun Chen, Email: zijun.chen@monash.edu}
	\footnotetext[1]{Shengkun Wu, Email: shengkunwu@foxmail.com }
	\maketitle
	
	{\noindent\small{\bf Abstract:}
		The Zakharov system in dimension $d=2,3$ is shown to have a local unique solution for any initial values in the space $H^{s} \times H^{l} \times H^{l-1}$, where a new range of regularity $(s, l)$ is given, especially at the line $s-l=-1$. The result is obtained mainly by the normal form reduction and the Strichartz estimates.}
	
	\vspace{1ex}
	{\noindent\small{\bf Keywords:}
		Zakharov; local well-posedness; normal form reduction; strichartz estimate}

\section{Introduction}
The initial-value problem for the Zakharov system is:
\begin{equation}\label{1.1}
	\left\{\begin{array}{l}
		i \dot{u}-\Delta u=n u,\\
		\ddot{n} / \alpha^{2}-\Delta n=-\Delta|u|^{2}, \\
		u(0, x)=u_{0}(x), n(0, x)=n_{0}(x), \dot{n}(0, x)=n_{1}(x).
	\end{array}\right.
\end{equation}
Here, $ u(t, x):  \mathbb{R} \times \mathbb{R}^{d} \longmapsto \mathbb{C}, n(t, x): \mathbb{R} \times \mathbb{R}^{d} \longmapsto \mathbb{R}.  $ $\alpha>0$ is called the ion sound speed. This problem arises in plasma physics.  Sufficiently regular solutions of $\eqref{1.1}$ satisfy the conservation of mass
\begin{equation}
	\int_{\mathbb{R}^{d}}|u|^{2} d x=\int_{\mathbb{R}^{d}}\left|u_{0}\right|^{2} d x
\end{equation}
and conservation of the Hamiltonian
\begin{equation}
	E=\int_{\mathbb{R}^{d}}|\nabla u|^{2}+\frac{\left|D^{-1} \dot{n}\right|^{2} / \alpha^{2}+|n|^{2}}{2}-n|u|^{2} d x,
\end{equation}
where $D:=\sqrt{-\Delta}$. The Zakharov system is typically studied as a Cauchy problem by prescribing initial data in the space
\begin{equation}
	\left(u_{0}, n_{0}, n_{1}\right) \in H^{s}\left(\mathbb{R}^{d}\right) \times H^{l}\left(\mathbb{R}^{d}\right) \times H^{l-1}\left(\mathbb{R}^{d}\right).
\end{equation}
In terms of $N:=n-i D^{-1} \dot{n} / \alpha$, we can change the system $\eqref{1.1}$ into first order as usual:
\begin{equation}\label{equivalent1.1}
	\left\{\begin{array}{l}
		\left(i \partial_{t}-\Delta\right) u=N u / 2+\bar{N} u / 2, \\
		\left(i \partial_{t}+\alpha D\right) N=\alpha D|u|^{2},\\
		u(0, x)=u_{0}(x), N(0, x)=N_{0}(x),
	\end{array}\right.
\end{equation}
where $\left(u_{0}, N_{0}\right) \in H^{s}\left(\mathbb{R}^{d}\right) \times H^{l}\left(\mathbb{R}^{d}\right)$.
Since the term $\bar{N} u$ makes no essential difference from $N u$, from here on out, we will replace the nonlinear term $N u / 2+\bar{N} u / 2$ with $Nu$. 

The system $\eqref{1.1}$ was introduced by Zakharov \cite{Zakharov1972} to depict the Langmuir turbulence in unmagnetized ionized plasma. By using Bourgain's $X^{s, b}$ space \cite{Bourgain1996}, Ginibre, Tsutsumi, and Velo \cite{Ginibre1997} proved the best known result of the local well-posedness (LWP) for the Zakharov system in some regular spaces $H^{s} \times H^{l} \times H^{l-1}$ with various $s, l$ in arbitrary space dimensions. Further well-posedness results were obtained in \cite{Pecher2005, Colliander2008} if $d=1$; in  \cite{Fang2009, Bejenaru2009} if $d=2$; in \cite{Bejenaru2011, Bejenaru2013}  if $d=3$; and in \cite{Bejenaru2015, Kato2017, Candy2020} if $d \geq 4$.

Normal form reduction is widely used to deal with the difficulty posed by the so-called “derivative loss” during the study of the Zakharov system. Employing this method in \cite{Guo2014}, Guo and Nakanishi  proved small energy scattering for the Zakharov system in $d=3$. Then, in \cite{GuoLee2014}, the generalized Strichartz estimate was used to obtain improvements. Moreover, in \cite{GuoNakanishi2014}, global dynamics below ground state energy were considered for the Klein-Gordon-Zakharov system. Recently, for $d=4$, Bejenaru, Guo, Herr, and Nakanishi \cite{Bejenaru2015} proved small data global well-posedness and scattering in a range of $(s, l)$. Additionally, the large data threshold result in \cite{GuoNakanishi2018} is shown to be restricted to radial data. 

In this article, we are particularly interested in the low regularity local well-posedness for the Zakharov system $\eqref{1.1}$ in $d=2, 3$ with $\left(u_{0}, n_{0}, n_{1}\right) \in H^{s} \times H^{l} \times H^{l-1}$. We will combine the normal form reduction with multilinear estimates, which easily follow from Littlewood-Paley decomposition, Coifman Meyer bilinear estimate, Strichartz estimates, and Sobolev inequalities. 

Using these multilinear estimates and the standard contraction mapping principle, we obtain two main results as follows:

\begin{theorem}\label{Theorem1}
	The Cauchy problem $\eqref{1.1}$ is locally well-posed in $ H^{s}\left(\mathbb{R}^{2}\right) \times H^{l}\left(\mathbb{R}^{2}\right)\\ \times H^{l-1}\left(\mathbb{R}^{2}\right)$, provided that
	\begin{equation}\label{result1}
		l \geq 0, \quad \max\left\{\frac{l+1}{2}, \ l-1\right\} \leq s \leq l+\frac{3}{2}.
	\end{equation} 
\end{theorem}
	\begin{theorem}\label{Theorem2}
The Cauchy problem $\eqref{1.1}$ is locally well-posed in $ H^{s}\left(\mathbb{R}^{3}\right) \times H^{l}\left(\mathbb{R}^{3}\right) \\ \times H^{l-1}\left(\mathbb{R}^{3}\right)$, provided that
		\begin{equation}\label{result2}
			l \geq 0, \quad \max\left\{\frac{l+1}{2}, \ l-1\right\} \leq s \leq l+\frac{5}{4}.
		\end{equation} 
	\end{theorem}
	
	\begin{remark}
	(1) To the best of our knowledge, the latest local well-posedness result for $d=2$ is showed in \cite{Bejenaru2009}, which shows LWP to be in a corner at $L^{2} \times H^{-1 / 2} \times H^{-3 / 2}$ and with large initial data. This is the lowest regularity at which one expects to prove LWP via contraction mapping. In \cite{Bourgain1996}, our result covers the range $l \geq 0,2 s \geq l+1, l \leq s \leq l+1$. By using normal form reduction, we are unable to go below or towards the left of the lowest point $(s, l)=\left(0,-1/2\right)$, but we can improve upon the difference in the $s$ and $l$ regularities. More precisely, our result improves upon the constraint $l \leq s \leq l+1$, which is illustrated in Figure $1(\mathrm{a})$.
	
	(2) The best known result for $d=3$ (to our knowledge) \cite{Bejenaru2011} covers the region $l>-1/2, l \leq$ $s \leq l+1,2 s>l+1/2$, whereas our result improves in the positive regularity region to that given by \eqref{result2}, which is illustrated in Figure $1(\mathrm{b})$.
	
	(3) We found the result in arXiv:2103.09259 \cite{Sanwal2021} when we have compiled this paper. We state that one cannot go beyond the line $l=0$ using the normal form method, and hence we are unable to reach the lowest point $(s, l)=\left(0,-1/2\right)$. Our result can cover the line $s-l=-1$ which was not covered in \cite{Sanwal2021}, however, the result in \cite{Sanwal2021} does give a “broader” region of LWP.
	\end{remark}

\begin{figure}[htp]
	\subfigure[$d=2$: grey region for \cite{Bourgain1996} and new region enclosed by red lines.]
	{	
		\centering
		\includegraphics[width=2in]{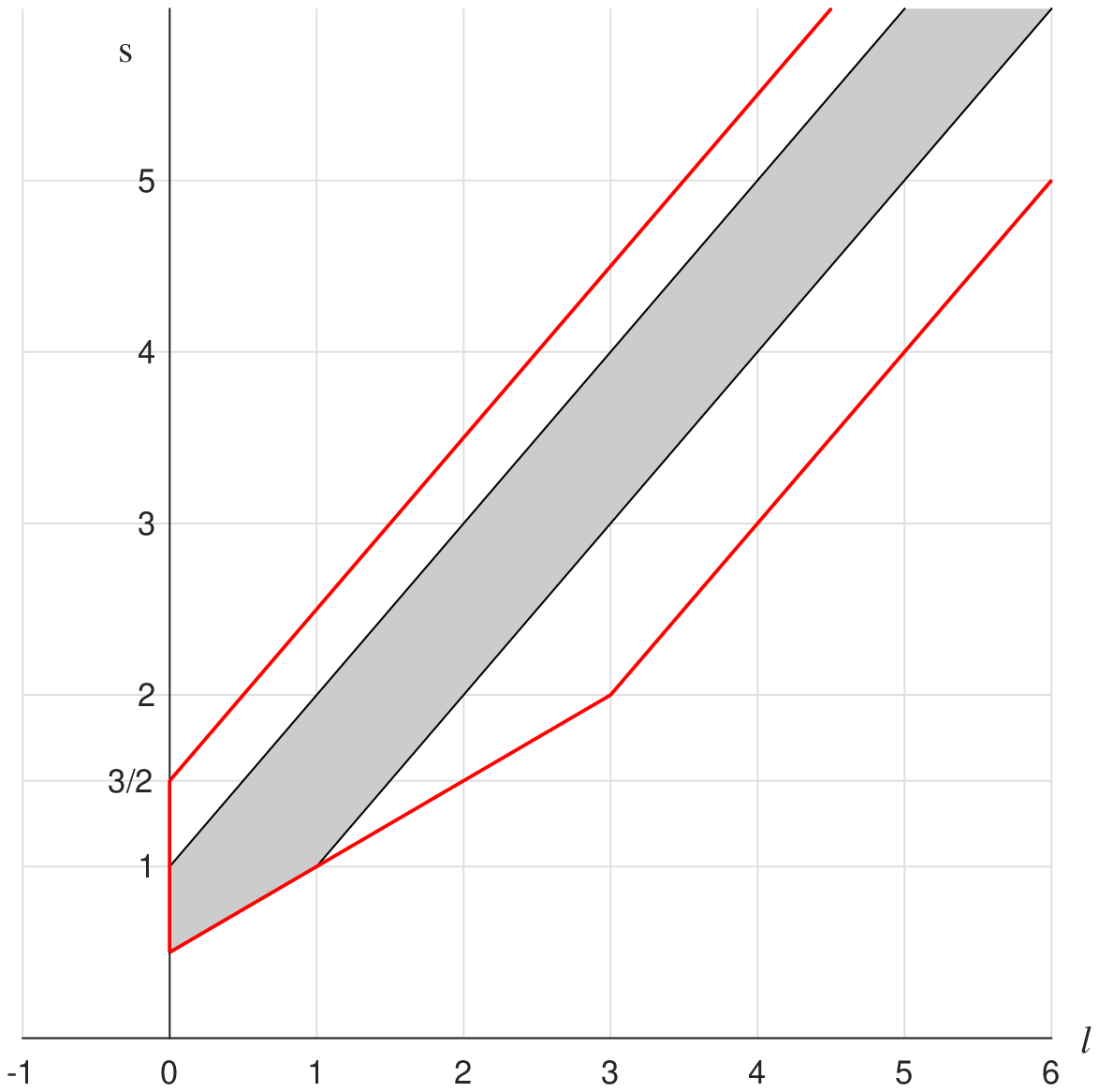}
		\label{fig:subfig:a}
}
	\subfigure[$d=3$: grey region for \cite{Bejenaru2011} and new region enclosed by red lines.]
	{
		\centering
		\includegraphics[width=2in]{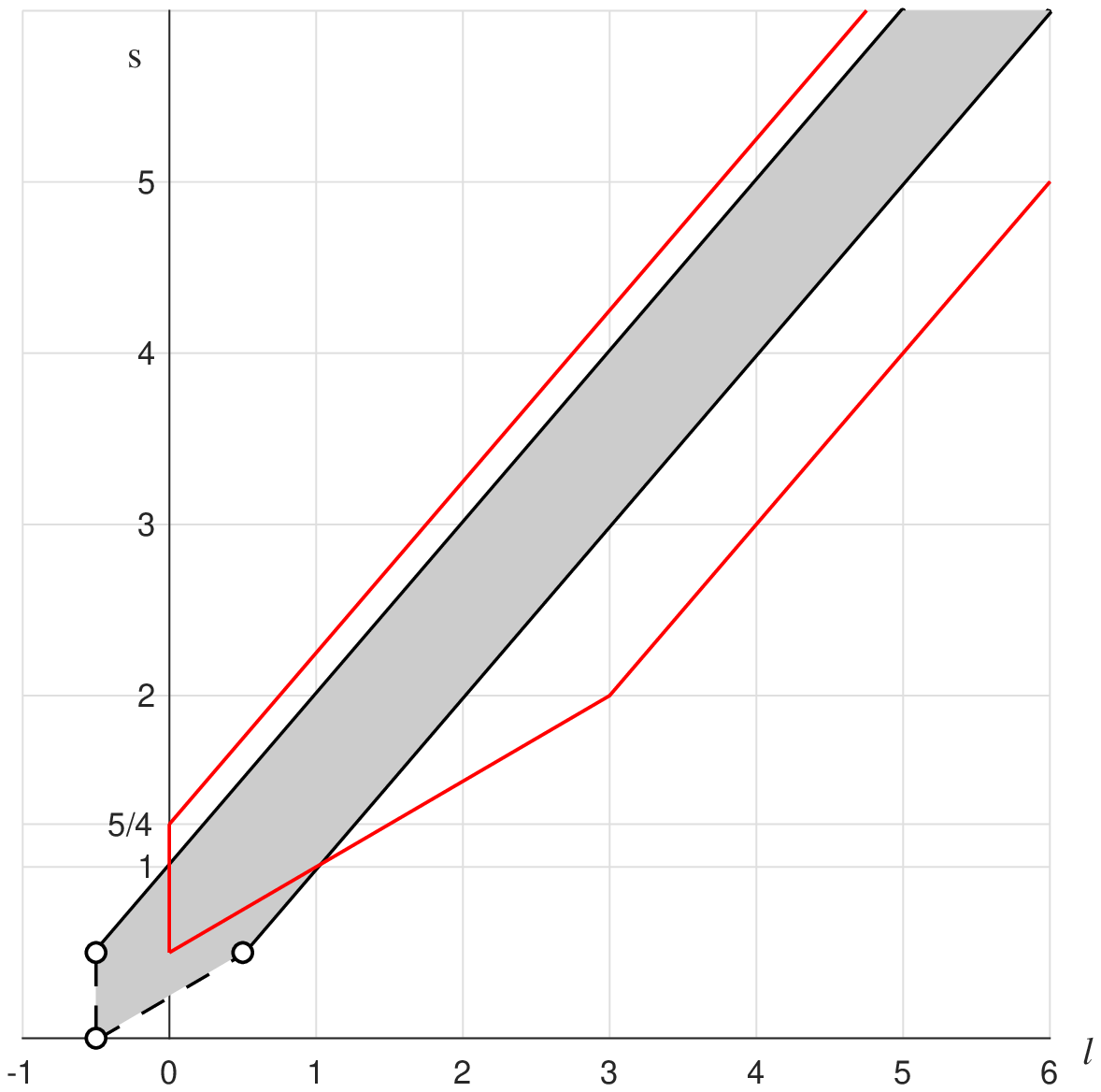}
		\label{fig:subfig:b}
}
	\caption{Region of regularity $(s, l)$.}
	\label{}
\end{figure}
\section{Notations and normal form}
This section is devoted to introducing some basic preparations. We use $S(t)$ and $W_{\alpha}(t)$ to respectively denote Schr\"{o}dinger and wave semigroup:
\begin{equation}
	S(t) \phi=\mathcal{F}^{-1} e^{i t \left|\xi\right|^{2}} \hat{\phi}, \quad W_{\alpha}(t) \phi=\mathcal{F}^{-1} e^{i \alpha t|\xi|} \hat{\phi}, \quad \hat{\phi}=\mathcal{F} \phi,
\end{equation}
where $\mathcal{F}$ denotes the spatial Fourier transform. Let $\eta_{0}: \mathbb{R}^{d} \rightarrow[0,1]$ denote a radial, smooth function supported in $\{|\xi| \leq 8 / 5\}$, which is equal to 1 in $\{|\xi| \leq 5 / 4\} .$ For $k \in \mathbb{Z}$, let $\chi_{k}(\xi)=\eta_{0}\left(\xi / 2^{k}\right)-\eta_{0}\left(\xi / 2^{k-1}\right)$ and
$\chi_{\leq k}(\xi)=\eta_{0}\left(\xi / 2^{k}\right) .$ Let $P_{k}$, $P_{\leq k}$ denote the operators on $L^{2}\left(\mathbb{R}^{d}\right)$ which are defined by $\widehat{P_{k} u}(\xi)=\chi_{k}(|\xi|) \widehat{u}(\xi)$, $\widehat{P_{\leq k} u}(\xi)=\chi_{\leq k}(|\xi|) \widehat{u}(\xi)$.

Let $s, l \in \mathbb{R}$ and $1 \leq p, q \leq \infty$. The standard homogeneous Besov and the  inhomogeneous Besov spaces are defined respectively by
\begin{equation}
	\begin{aligned}
	\|f\|_{\dot{B}_{p, q}^{s}}=\left\| 2^{k s}\left\|P_{k} f\right\|_{p}\right\|_{\ell_{k}^{q}(\mathbb{Z})}, \
	 \|f\|_{B_{p, q}^{s}}=\|P_{\leq 0}f\|_p+\left(\sum_{k=1}^{\infty} 2^{k s q}\left\|P_{k} f\right\|_{p}^{q}\right)^{1 / q}.
	\end{aligned}
\end{equation}
Here, we simply write $B_{p}^{s}=B_{p, 2}^{s}$, $H^{s}=B_{2}^{s}$ and the same abbreviations are defined in the homogeneous Besov spaces. 

Next, we turn to introduce the normal form reduction.  For two functions $N$, $u$ and a fixed $K \in \mathbb{N}, K \geq 5$,  $Nu$ is divided into
\begin{equation}
	N u =(N u)_{H H}+(N u)_{L H}+(N u)_{H L} .
\end{equation}
Here $(N u)_{H H}$ and $(N u)_{L H}$ are respectively denoted as:
\begin{equation*}
	(N u)_{H H}:=\sum_{k_{1}, k_{2} \in \mathbb{Z} \atop \left|k_{1}-k_{2}\right| \leq K-1} P_{k_{1}} N P_{k_{2}} u,\
	(N u)_{L H}:=\sum_{k \in \mathbb{Z}} P_{\leq k-K} N P_{k} u, 
\end{equation*}
and 
$$ (N u)_{H L}:=\sum_{k \in \mathbb{Z}} P_{k} N P_{\leq k-K} u.$$
Denote $\beta \geq K+|\log_{2}\alpha|$. We also define
\begin{equation}
	\begin{aligned}
&(N u)_{\alpha L}:= \sum_{k \in \mathbb{Z} \atop k \leq \beta} P_{k} N P_{\leq k-K} u, \ (N u)_{L \alpha}:=(u N)_{\alpha L},\\
&(N u)_{X L}:=\sum_{k \in \mathbb{Z} \atop k> \beta} P_{k} N P_{\leq k-K} u,  \ (N u)_{L X}:=(u N)_{X L},
\end{aligned}
\end{equation}
 so that
\begin{equation}
	\begin{aligned}
		(N u)_{H L}=(N u)_{\alpha L}+(N u)_{X L}.
	\end{aligned}
\end{equation}
For any such index $*=H H, H L, \alpha L,$ etc., the bilinear multiplier is denoted by
\begin{equation}
	\mathcal{F}(N u)_{*}=\int \mathcal{P}_{*} \hat{N}(\xi-\eta) \hat{u}(\eta) d \eta.
\end{equation}

In view of Duhamel's formula and the normal form reduction as in \cite{Bejenaru2015} and \cite{Guo2014}, the system $\eqref{equivalent1.1}$ can be rewritten as the following integral equations:
\begin{equation}\label{u}
	\begin{aligned}
		u=& S(t) u_{0}-S(t) \Omega(N, u)(0)+\Omega(N, u)(t)\\
		&-i \int_{0}^{t} S(t-s)(N u)_{L H+H H+\alpha L}(s) d s\\
		&-i \int_{0}^{t} S(t-s) \Omega\left(\alpha D|u|^{2}, u\right)(s) d s-i \int_{0}^{t} S(t-s) \Omega(N, N u)(s) d s
	\end{aligned}
\end{equation}
for $\Omega(f, g)=\mathcal{F}^{-1} \int \mathcal{P}_{X L} \frac{\hat{f}(\xi-\eta) \hat{g}(\eta)}{-|\xi|^{2}+\alpha|\xi-\eta|+|\eta|^{2}} d \eta.$
\begin{equation}\label{N}
	\begin{aligned}
		N=& W_{\alpha}(t) N_{0}-W_{\alpha}(t) D \tilde{\Omega}(u, u)(0)+D \tilde{\Omega}(u, u)(t)\\
		&-i\int_{0}^{t} W_{\alpha}(t-s) \alpha D(u \bar{u})_{H H+\alpha L+L \alpha} (s)d s \\
		&-i\int_{0}^{t} W_{\alpha}(t-s)(D \tilde{\Omega}(N u, u)+D \tilde{\Omega}(u, N u))(s) d s
	\end{aligned}
\end{equation}
for $\tilde{\Omega}(f, g)=\mathcal{F}^{-1} \int \mathcal{P}_{X L+L X} \frac{\alpha \hat{f}(\xi-\eta) \hat{\bar{g}}(\eta)}{|\xi-\eta|^{2}-|\eta|^{2}-\alpha|\xi|} d \eta.$

Therefore, the equations, after normal form reduction, can be read as
\begin{equation}\label{final equation}
	\begin{array}{l}
		\left(i \partial_{t}+D^{2}\right)(u-\Omega(N, u))=(N u)_{L H+H H+\alpha L}+\Omega\left(\alpha D|u|^{2}, u\right)+\Omega(N, N u) ,\\
		\left(i \partial_{t}+\alpha D\right)(N-D \tilde{\Omega}(u, u))=\alpha D|u|_{H H+\alpha L+L \alpha}^{2}+D \tilde{\Omega}(N u, u)+D \tilde{\Omega}(u, N u).
	\end{array}
\end{equation}

We will use the following Strichartz estimates for the Schr\"{o}dinger and the wave equation.
\begin{lemma} (Strichartz estimates, see \cite{Keel1998})\label{strichartz}
	For  any functions $\phi(x)$ and $f(t, x)$, 
	
	(1) if $(q, r)$ and $(\tilde{q}, \tilde{r})$ both satisfy the Schr\"{o}dinger-admissible condition:
	\begin{equation}
		2\leq q, r \leq \infty, \frac{1}{q}=\frac{d}{2}\left(\frac{1}{2}-\frac{1}{r}\right), (q, r, d) \neq(2, \infty, 2),
		\end{equation}
	where
	\begin{equation}
	\begin{cases}2 \leq r<\infty, & d=2, \\ 2 \leq r \leq 6, & d = 3,\end{cases}
	\end{equation}
	then
	\begin{equation}
		\|S(t) \phi\|_{L_{t}^{\infty} H_{x}^{s} \cap L_{t}^{q} B_{r}^{s}} \lesssim\|\phi\|_{H^{s}},
	\end{equation}
	\begin{equation}\label{s-strichartz}
		\left\|  \int^{t}_{0}S(t-s)f(s) ds \right\|_{L_{t}^{\infty} H_{x}^{s} \cap L_{t}^{q} B_{r}^{s}}
		\lesssim \left\| f \right\|_{L^{\tilde{q}^{\prime}}_{t}B^{s}_{\tilde{r}^{\prime}}}.
	\end{equation}
	
	(2) If $(\tilde{q}, \tilde{r})$ satisfies the wave-admissible condition:
		\begin{equation}
		2\leq \tilde{q}, \tilde{r} \leq \infty, \frac{1}{\tilde{q}}=\frac{d-1}{2}\left(\frac{1}{2}-\frac{1}{\tilde{r}}\right), (\tilde{q}, \tilde{r}, d) \neq(2, \infty, 3),
	\end{equation}
	where
	\begin{equation}
		\begin{cases}2 \leq \tilde{r} \leq \infty, & d=2, \\ 2 \leq \tilde{r}<\infty, & d=3, \end{cases}
	\end{equation}
	then
	\begin{equation}
		\|W_{\alpha}(t) \phi\|_{L_{t}^{\infty} H_{x}^{l}} \lesssim\|\phi\|_{H^{l}},
	\end{equation}
	\begin{equation}
		\left\|  \int^{t}_{0}W_{\alpha}(t-s)f(s)ds \right\|_{L_{t}^{\infty} H_{x}^{l}} \lesssim \left\| f \right\|_{L^{\tilde{q}^{\prime}}_{t}B^{l}_{\tilde{r}^{\prime}}}.
	\end{equation}
\end{lemma}

\section{Local well-posedness for $d=3$}
In this section, we consider the local well-posedness for $d=3$. According to Lemma $\ref{strichartz}$, we define the following spaces:
\begin{equation}\label{Xs}
	\begin{aligned}
	u \in X^{s}:= &C \left([-T, T];H^{s}\left(\mathbb{R}^{3}\right)  \right) \cap L^{\infty}\left([-T,T];H^{s}\left(\mathbb{R}^{3}\right) \right)\\
	&\cap L^{q}\left([-T,T];B^{s}_{\frac{6q}{3 q - 4}}\left(\mathbb{R}^{3}\right) \right),
	\end{aligned}
\end{equation}
and
\begin{equation}\label{Yl}
	N \in Y^{l}:= C \left([-T,T];H^{l}\left(\mathbb{R}^{3}\right) \right) \cap L^{\infty}\left([-T,T];H^{l}\left(\mathbb{R}^{3}\right) \right),
\end{equation}
for $2< q \leq \infty$ and $T>0$. Finally, we determine $q=8/3$, and so \eqref{Xs} is rewritten as 
\begin{equation}\label{Xs-r}
	\begin{aligned}
	u \in X^{s}:= &C \left([-T, T];H^{s}\left(\mathbb{R}^{3}\right)  \right) \cap L^{\infty}\left([-T,T];H^{s}\left(\mathbb{R}^{3}\right) \right)\\
	&\cap L^{8/3}\left([-T,T];B^{s}_{4}\left(\mathbb{R}^{3}\right) \right).
	\end{aligned}
\end{equation}
The $q$ selection will be shown in the following subsections.

We now turn to prove multilinear estimates for the nonlinear terms of $\eqref{final equation}$ in the above spaces. For brevity, we selectively keep some dependence of the constants on $(s, l)$ and $\beta$.
\subsection{Quadratic terms}
Since we use the contraction mapping principle, $u$ and $N$ must have closed $X^{s}$ and $Y^{l}$ estimates, respectively. In terms of  H\"{o}lder's inequality for quadratic terms, we apply the following Strichartz estimates:
\begin{equation}\label{qs1}
	\left\|\int^{t}_{0}S(t-s)(Nu)_{L H+H H+\alpha L}(s) ds \right\|_{L_{t}^{\infty}H^{s} \cap L_{t}^{8/3}B^{s}_{4}} \lesssim \left\| (Nu)_{L H+H H+\alpha L}  \right\|_{L^{8/5}_{t}B^{s}_{4/3}},
\end{equation}
and
	\begin{equation}\label{qs2}
	\left\|\int_0^t W_{\alpha} (t-s) D(uv) _{HH + \alpha L+L\alpha }(s)ds\right\| _{L_t^{\infty}H^l} \lesssim \left\|D(uv)_{HH+\alpha L+L\alpha } \right\|_{L_{t}^{4/3}B^l_{4/3}}.
\end{equation}
In addition,  admissible conditions yield that the range of $q$ in \eqref{Xs} should be $2< q <4$ if we directly substitute spaces \eqref{Xs} and \eqref{Yl} into \eqref{qs1} and \eqref{qs2}, respectively.
\begin{lemma}\label{Quadratic terms} 
	
	(1) If $s, l \geq 0$, then for any $N(x)$ and $u(x)$,
	\begin{equation}\label{Quadratic 1}
		\left\|(N u)_{L H+HH}\right\|_{L^{8/5}_{t}B^{s}_{4/3}} \lesssim T^{1/4}\|N\|_{L_{t}^{\infty}H^{l}}\|u\|_{L_{t}^{8/3}B_{4}^{s}}.
	\end{equation}
	\begin{equation}\label{Quadratic 11}
		\left\|(N u)_{\alpha L}\right\|_{L^{8/5}_{t}B^{s}_{4/3}} \lesssim T^{1/4}C(\beta)\|N\|_{L_{t}^{\infty}H^{l}}\|u\|_{L_{t}^{q}B_{4}^{s}}.
	\end{equation}
	
	(2) If $2s \geq l+1$, then for any $u(x)$ and $v(x)$,
	\begin{equation}\label{Quadratic 2}
		\left\|D(u v)_{H H }\right\|_{L_{t}^{4/3}B^l_{4/3}} \lesssim T^{3/8}\|u\|_{L^{\infty}_{t}H^{s}}\|v\|_{L^{8/3}_{t}B_{4}^{s}}.
	\end{equation}
	\begin{equation}\label{Quadratic 22}
		\left\|D(u v)_{\alpha L+L \alpha }\right\|_{L_{t}^{4/3}B^l_{4/3}} \lesssim T^{3/8}C(\beta)\|u\|_{L^{\infty}_{t}H^{s}}\|v\|_{L^{8/3}_{t}B_{4}^{s}}.
	\end{equation}
\end{lemma}
\begin{proof}
	For (1), we begin with the $(Nu)_{LH}$ part. 
	Clearly, the norm for $t$ follows directly from H\"{o}lder's inequality. By $(N u)_{L H}=\sum_{j \in \mathbb{Z}} P_{\leq j-K} N P_{j} u$ and analyzing the support deduce that
	\begin{equation}
		\begin{aligned}
			\left\|P_{k}(N u)_{L H}\right\|_{L^{4 / 3}} \lesssim \sum_{j=k-2}^{k+2}\left\|\left(P_{\leq j-K} N\right)\left(P_{j} u\right)\right\|_{L^{4 / 3}} \lesssim \sum_{j=k-2}^{k+2}\|N\|_{L^{2}}\left\|P_{j} u\right\|_{L^{4}},
		\end{aligned}
	\end{equation}
then we obtain
	\begin{equation}\label{LH}
		\begin{aligned}
			\left\|(Nu)_{LH} \right\| _{B_{4/3}^s} \lesssim  \left\|2^{ks} \sum_{j=k-2}^{k+2}\|N\|_{L^{2}}\left\|P_{j} u\right\|_{L^{4}}\right\|_{\ell^2_k} \lesssim \left\|N \right\|_{H^{l}} \left\|u\right\|_{B^{s}_{4}},
		\end{aligned}
	\end{equation}
	where $H^{l}\left(\mathbb{R}^{3}\right) \subset L^{2}\left(\mathbb{R}^{3}\right)$ for  $l \geq0.$ 
	
	Similarly, the estimate $\left\|(N u)_{HH}\right\|_{L_{t}^{8/5}B_{4/3}^{s}} \lesssim T^{1/4}\|N\|_{L_{t}^{\infty}H^{l}}\|u\|_{L_{t}^{8/3}B_{4}^{s}}$ is a straightforward result as the proof given in $\eqref{LH}$.
	
	In particular, concerning the fact that  $(Nu)_{\alpha L}$ is in the support of $k \sim j \lesssim \beta$, we can conclude 
	\begin{equation}
		\left\|(Nu)_{\alpha L} \right\| _{L_{t}^{8/5}B^s_{4/3}}
		\lesssim T^{1/4} 2^{2\beta s} \left\|N \right\|_{L_{t}^{\infty}H^l}\left\|u \right\|_{L_{t}^{8/3}B^s_{4}}, 
	\end{equation}
	where we require $l\geq 0$ and $s \geq 0.$ Thus, (1) of  Lemma $\ref{Quadratic terms}$ is proven. 
	
	The proof for (2) follows in a similar manner.
\end{proof}
\begin{remark}
	In view of the admissible conditions in Lemma $\ref{strichartz}$, we can verify that for the proof of (1), the norm of $u$ can only adopt the space $L^{8/3}_{t}B_{4}^{s}$, instead of $L_t^{\infty}H^s$. The same verification is true for the space selections in the proof for (2).
\end{remark}
\subsection{Boundary terms}
We now turn to the boundary terms estimates. However, prior to that, we will review the crucial formulas given in \cite{Bejenaru2015}. Take $\Omega(N,u)$ as an example. By Coifman Meyer type bilinear multiplier and Bernstein estimates, we obtain
\begin{equation}\label{crucial}
	\begin{aligned}
		&\left\|P_{k}\langle D\rangle D\Omega(N, u) \right\|_{L^p} \lesssim \sum_{k_0=k-2}^{k+2}\left\|P_{k_0}N\right\|_{L^{p_1}} \sum_{k_{1} \leq k_0-K \atop k_{0}>\beta}\left\|P_{k_1}u\right\|_{L^{p_2}}\\
		&\lesssim \sum_{k_0=k-2}^{k+2}2^{k_0^{+}\left(\frac{3}{q_1}-\frac{3}{p_1}-l\right)}\left\|P_{k_{0}}N\right\|_{B^{l}_{q_1}} \sum_{k_{1} \leq k_{0}-K \atop k_{0}>\beta}2^{k_1\left(\frac{3}{q_2}-\frac{3}{p_2}\right)-k_{1}^{+}s}\left\|P_{k_1}u\right\|_{B_{q_2}^{s}}
	\end{aligned}
\end{equation}
with
\begin{equation}\label{crucial condition}
	\left\{\begin{array}{l}
		k, k_{0}, k_{1}\in \mathbb{Z}, \\
		p, p_{1}, p_{2}, q_{1}, q_{2} \in[1, \infty],\\
		1 / p=1 / p_{1}+1 / p_{2}, \\
		p_{1} \geq q_{1}, p_{2} \geq q_{2},
	\end{array}\right.
\end{equation}
where $k_{i}^{+}:=\max(k_{i}, 0)$, using $P_{\leq 0} B_{p}^{s} \subset \dot{B}_{p, \infty}^{0}$ for the lower frequency component. The same estimate also holds for the bilinear operator $\tilde{\Omega}$.

\begin{lemma}\label{Boundary terms} For any $\theta_{i}(s, l) \geq 0$, as well as functions $N(x)$, $u(x)$,  and $v(x)$, we have the following:
	
	(1) If $l \geq -1/2$, $s \leq l+2$ and $(s, l) \neq \left(3/2, -1/2\right)$, then
	\begin{equation}\label{boundary1}
		\|\Omega(N, u)\|_{L_{t}^{\infty}H^{s}} \lesssim 2^{- \beta \theta_{1}}\|N\|_{L_{t}^{\infty}H^{l}}\|u\|_{L_{t}^{\infty}H^{s}}.
	\end{equation}
	
	(2) If $l \geq -1/2$, $s \leq l+5/4$ and $(s, l) \neq \left(3/4, -1/2\right)$, then
	\begin{equation}\label{boundary2}
		\|\Omega(N, u)\|_{L_{t}^{8/3}B^{s}_{4}} \lesssim 2^{- \beta \theta_{2} }\|N\|_{L_{t}^{\infty}H^{l}}\|u\|_{L_{t}^{8/3}B^{s}_{4}}.
	\end{equation}
	
	(3) If $s \geq \text{max} \left\{l-1,  l/2+1/4\right\}$ and $(s, l) \neq \left(3/2, 5/2\right)$, then
	\begin{equation}\label{boundary3}
		\|D \tilde{\Omega}(u, v)\|_{L_{t}^{\infty}H^{l}} \lesssim 2^{- \beta \theta_{3}}\|u\|_{L_{t}^{\infty}H^{s}}\|v\|_{L_{t}^{\infty}H^{s}}.
	\end{equation}
\end{lemma}
\begin{proof}
	Here we prove $\eqref{boundary2}$ as an example. In view of $\eqref{crucial}$, a short computation with $(p,p_1,p_2,q_1,q_2)= \left(4 , 4, \infty, 2, 4\right)$ shows that 
	\begin{equation}
		\begin{aligned}
			\left\|\Omega(N,u)\right\|_{B^s_{4}}
			&\lesssim \left\|2^{k^{+}\left(s-l-5/4\right)} \left\|P_{k}N \right\|_{H^{l}}\sum\limits_{k_1 \leq k-K \atop k > \beta} 2^{3k_1/4-k_{1}^{+}s} \left\|P_{k_1}u \right\|_{B^s_{4}}\right\|_{\ell^2_k}.
		\end{aligned}
	\end{equation}
	It suffices to discuss
	\begin{equation}\label{summation}
	2^{k^{+}\left(s-l-5/4\right)}\sum\limits_{k_1 \leq k-K \atop k > \beta} 2^{3k_{1}/4-k_{1}^{+}s}.
	\end{equation}
   First, the summation over $ k_1 \leq k-K$ deduces
	\begin{equation}
		k_{1} \leq 0 \Rightarrow \sum\limits_{k_1 \leq k-K} 2^{3k_{1}/4-k_{1}^{+}s}\lesssim 1,
			\end{equation}
		and
			\begin{equation}
		0 < k_1 \leq k-K \Rightarrow \sum\limits_{k_1 \leq k-K} 2^{3k_{1}/4-k_{1}^{+}s}
		\left\{
		\begin{array}{ll}
			k, \quad & s=3/4,\\
			1 ,\quad &s>3/4,\\
			2^{k\left(3/4-s\right)} ,\quad &s<3/4.
		\end{array}\right.
	\end{equation}
	Then, by a summation over $k$, we obtain for $\theta_{2}(s, l) \geq 0$,  $\eqref{summation} \lesssim 2^{-\beta\theta_{2}(s, l)}$ if $l \geq -1/2$, $s \leq l+5/4$ and $(s, l) \neq \left(3/4, -1/2\right)$. This and $\left\|P_{k} N\right\|_{H^{l}} \in \ell_{k}^{2}$, together with H\"{o}lder's inequality for $t$ lead to \eqref{boundary2}.
	
	Since the proofs are all similar, we only give the choices of  $(p, p_{1}, p_{2}, q_{1}, q_{2})$ for \eqref{boundary1} and \eqref{boundary3}, both of which are $\left(2, 2,\infty, 2, 2\right)$. 
\end{proof}
\begin{remark}
	
	(1) We now explain how to determine $p, p_{1}, p_{2}, q_{1}, q_{2}$ for $\eqref{boundary2}$. In order to guarantee that $u$ and $N$, here mainly corresponding to $t$, respectively have closed $X^{s}$ and $Y^{l}$ estimates when using H\"{o}lder's inequality, we shall be forced to take $q_{1}=2$, and $q_{2}=4.$ In terms of the confined conditions in $\eqref{crucial condition}$, $p_{1}$ should be $\max\{p, q_{1}\}=4$. Furthermore, $1/p=1/p_1+1/p_2$ derives $p_2=\infty$. The same method of parameter selection also holds for \eqref{boundary1} and \eqref{boundary3}.
	
	(2) The range of $q$ here should be $2< q \leq \infty$ if spaces \eqref{Xs} and \eqref{Yl} are directly used for discussion.
	\end{remark}
\subsection{Cubic terms}
In this section, we deal with the cubic terms. 
\begin{lemma}\label{Cubic terms}  For any $M(x), N(x), u(x), v(x)$, we have the following:
	
	(1) If  $l \geq -1/4$, $s \leq \min\{l+2, 2l+5/4\}$ and $(s, l) \neq \left(7/4, -1/4\right), \left(11/4, 3/4\right),\\ \left(3/4, -1/4\right)$, then
	\begin{equation}\label{MNu}
		\|\Omega(M, N u)\|_{L^{8/5}_{t}B_{4/3}^{s}} \lesssim T^{1/4}\|M\|_{L^{\infty}_{t}H^{l}}\|N\|_{L^{\infty}_{t}H^{l}}\|u\|_{L^{8/3}_{t}B_{4}^{s}}.
	\end{equation}
	
	(2) If $s \geq 1/2,$ then
	\begin{equation}\label{uvw}
		\|\Omega(D(u v), w)\|_{L^{8/5}_{t}B_{4/3}^{s}} \lesssim T^{1/4} \left\|\|u\|_{H^{s}}\|v\|_{B_{4}^{s}}+\|u\|_ {B_{4}^{s}}\|v\|_{H^{s}}\right\|_{L_{t}^{8/3}}\|w\|_{L_{t}^{\infty}H^{s}}.
	\end{equation}
	
	(3) If $s \geq 1/4$, $ l \leq \min\{s+1, 2s+1/4\}$ and $(s, l) \neq \left(1/4, 3/4\right), \left(3/4, 7/4\right)$, 
	then 
	\begin{equation}
		\begin{aligned}
		&\|D \tilde{\Omega}(N u, v)\|_{L^{1}_{t}H^{l}}+\|D \tilde{\Omega}(v, N u)\|_{L^{1}_{t}H^{l}}\\ &\lesssim T^{1/4}\|N\|_{L^{\infty}_{t}H^{l}}\|u\|_{L^{8/3}_{t}B_{4}^{s}}\|v\|_{L^{8/3}_{t}B_{4}^{s}}.
		\end{aligned}
	\end{equation}
\end{lemma}
\begin{proof}
	It suffices to prove (1) as an example. We obtain the following from $\eqref{crucial}$:
	\begin{equation}
		\begin{aligned}
			\left\|\Omega (M,Nu) \right\|_{B^s_{4/3}}
			&\lesssim \left\|2^{k(s-2)}\left\|P_{k}M\right\|_{L^{p_1}} \sum_{ k_1 \leq k-K\atop k > \beta}\left\|P_{k_1}(Nu)\right\|_{L^{\sigma}}\right\|_{\ell_{k}^{2}}.
		\end{aligned}
	\end{equation}
	Furthermore, we obtain three cases:
	\begin{align*}
		\left\|P_{k_1}(Nu)\right\|_{L^{\sigma}} \lesssim&\sum_{k_2=k_1-2}^{k_1+2} \left\|P_{k_2}u\right\|_{L^{p_2}} \sum_{k_3 \leq k_2-K}\left\|P_{k_3}N\right\|_{L^{p_3}} &\textbf{Case1}\\
		&+\sum_{k_2=k_1-2}^{k_1+2} \left\|P_{k_2}u\right\|_{L^{p_2}} \sum_{|k_3-k_2| \leq K-1}\left\|P_{k_3}N\right\|_{L^{p_3}} &\textbf{Case2}\\
		&+\sum_{k_3=k_1-2}^{k_1+2} \left\|P_{k_3}N\right\|_{L^{p_3}} \sum_{k_2 \leq k_3-K}\left\|P_{k_2}u\right\|_{L^{p_2}}, &\textbf{Case3}
	\end{align*}
where
\begin{equation}\label{crucial condition2}
	\left\{\begin{array}{l}
		k, k_{i}\ \in \mathbb{Z}, \\
	    p_{i}, q_{i} \in[1, \infty],\\
		3 / 4 =1 / p_{1}+1 / \sigma, 1 / \sigma=1 / p_{2}+1 / p_{3},\\
		p_{i} \geq q_{i}.
	\end{array}\right.
\end{equation}
		
	$ \textbf{Case 1:}$ For $k_3 \lesssim k_2 \sim k_1 \lesssim k $, concerning \eqref{crucial}, we have 
	\begin{equation}\label{case11}
		\begin{aligned}
			&2^{ks}\left\|P_k\Omega(M, Nu) \right\|_{L^p(\text{case 1})}\\
			&\lesssim 2^{k^{+}\left(s-2+\frac{3}{q_1}-\frac{3}{p_1}-l\right)}
			\sum_{k_1 \lesssim k}  2^{k_1\left(\frac{3}{q_2}-\frac{3}{p_2}\right)-k_{1}^{+}s}
			 \sum_{ k_3 \lesssim k_{1}}2^{k_3\left(\frac{3}{q_3}-\frac{3}{p_3}\right)-k_{3}^{+}l}\\ &\quad\times\left\|P_{k}M\right\|_{B^{l}_{q_1}}\left\|P_{k_1}u\right\|_{B^{s}_{q_2}}\left\|P_{k_3}N\right\|_{B_{q_3}^{l}}.
		\end{aligned}
	\end{equation}
First considering $\left(k_{1}>0, k_{3}>0\right)$, we obtain 
	\begin{equation}\label{case1}
		\begin{aligned}
			\eqref{case11}
			\lesssim 2^{k\left(s-2-l\right)}\sum_{ k_1 \lesssim k} 2^{-k_{1}s}\sum_{k_3 \lesssim k_{1}} 2^{k_{3}\left(3/2-l\right)} \left\|P_{k}M\right\|_{H^{l}}\left\|P_{k_1}u\right\|_{B^{s}_{4}}\left\|P_{k_3}N\right\|_{H^{l}}
		\end{aligned}
	\end{equation}
with $ (p,  \sigma, p_1, p_2, p_3, q_1, q_2, q_3) = \left(4/3,  4, 2, 4, \infty, 2, 4, 2\right)$.
	The summation over $k_3$ and $k_1$ deduces
	\begin{equation}\label{bounded1}
		l<3/2:
		\left\{\begin{array}{ll}
			-s+3/2-l>0\Rightarrow 2^{k\left(-1/2-2l\right)},\\
			-s+3/2-l<0\Rightarrow2^{k\left(-l+s-2\right)},\\
			-s+3/2-l=0\Rightarrow k2^{k\left(-l+s-2\right)},
		\end{array}\right.
	\end{equation}
	\begin{equation}\label{bounded2}
		l>3/2:
		\left\{\begin{array}{ll}
			s<0\Rightarrow 2^{k\left(-l-2\right)},\\
			s>0\Rightarrow2^{k\left(-l+s-2\right)},\\
			s=0\Rightarrow k2^{k\left(-l+s-2\right)},
		\end{array}\right.
	\end{equation}
	and
	\begin{equation}\label{bounded3}
		l=3/2:
		\left\{\begin{array}{ll}
			s<0\Rightarrow k2^{k\left(-l-2\right)},\\
			s>0\Rightarrow 2^{k\left(-l+s-2\right)},\\
			s=0\Rightarrow k2^{k\left(-l+s-2\right)}.
		\end{array}\right.
	\end{equation}
By summation over $k$, we can see  $l<3/2$ is fine if $l \geq -1/4$, $s \leq l+2$ and $(s, l) \neq \left(7/4, -1/4\right)$; $l>3/2$ is fine if $s \leq l+2$; and $l=3/2$ is fine if  $s \leq l+2$. To sum up, the summation of \eqref{case1} is bounded provided that $l \geq -1/4$, $s \leq l+2$ and $(s, l) \neq \left(7/4, -1/4\right)$.

For $\left(k_{1}>0, k_{3}\leq0\right)$ and $\left(k_{1}\leq 0\right)$ meanwhile, it is clear that the boundedness results can be contained by $\left(k_{3}>0, k_{1}>0\right)$. The same discussion is true for $\bf{Case 2}$ and $\bf{Case 3}$, hence we will omit part of proof.

$\textbf{Case 2:}$  For $k_3 \sim k_2 \sim k_1 \lesssim k $. As in \textbf{Case 1}, we choose 
$$ (p,  \sigma, p_1, p_2, p_3, q_1, q_2, q_3)= \left(4/3, 4, 2, 4, \infty, 2, 4, 2\right).$$
Differently however, it suffices to discuss
$2^{k\left(s-2-l\right)}\sum_{ k_1 \lesssim k} 2^{k_{1}\left(-s+3/2-l\right)}$ $\left(k_{1}>0\right)$, which is bounded by
\begin{equation}\label{bounded4}
	\left\{\begin{array}{ll}
		-s+3/2-l>0\Rightarrow 2^{k\left(-1/2-2l\right)},\\
		-s+3/2-l<0\Rightarrow2^{k\left(-l+s-2\right)},\\
		-s+3/2-l=0\Rightarrow k2^{k\left(-l+s-2\right)},
	\end{array}\right.
\end{equation}
provided that $l \geq -1/4$, $s \leq l+2$ and $(s, l) \neq \left(7/4, -1/4\right)$.

$\textbf{Case 3:}$ For $k_2 \lesssim k_3 \sim k_1 \lesssim k $, the boundedness is relevant to
\begin{equation}
	\begin{aligned}
		&2^{ks}\left\|P_k\Omega(M, Nu) \right\|_{L^p(\text{case 3})}\\
		&\lesssim 	2^{k^{+}\left(s-2+\frac{3}{q_1}-\frac{3}{p_1}-l\right)}\sum_{ k_1 \lesssim k} 2^{k_{1}\left(\frac{3}{q_3}-\frac{3}{p_3}\right)-k_{1}^{+}l}\sum_{k_2 \lesssim k_{1}} 2^{k_{2}\left(\frac{3}{q_2}-\frac{3}{p_2}\right)-k_{2}^{+}s}\\
		&\quad \times\left\|P_{k}M\right\|_{B^{l}_{q_1}}\left\|P_{k_1}N\right\|_{B^{l}_{q_3}}\left\|P_{k_2}u\right\|_{B_{q_2}^{s}}.
	\end{aligned}
\end{equation}
By summation over $(k_1>0, k_{2}>0)$, we obtain with $(p, \sigma, p_1, p_2, p_3, q_1, q_2, q_3) = \left(4/3, 4, 2, \infty, 4, 2, 4, 2\right)$ 
\begin{equation}\label{bounded31}
	s<3/4:
	\left\{\begin{array}{ll}
		-s+3/2-l>0\Rightarrow 2^{k\left(-1/2-2l\right)},\\
		-s+3/2-l<0\Rightarrow 2^{k\left(s-2-l\right)},\\
		-s+3/2-l=0\Rightarrow k2^{k\left(s-2-l\right)},
	\end{array}\right.
\end{equation}
\begin{equation}\label{bounded32}
	s>3/4:
	\left\{\begin{array}{ll}
		3/4-l>0\Rightarrow 2^{k\left(-5/4-2l+s\right)},\\
		3/4-l<0\Rightarrow 2^{k\left(	s-2-l\right)},\\
		3/4-l=0\Rightarrow k2^{k\left(	s-2-l\right)},
	\end{array}\right.
\end{equation}
and
\begin{equation}\label{bounded33}
	s=3/4:
	\left\{\begin{array}{ll}
		3/4-l>0\Rightarrow k2^{k\left(-5/4-2l+s\right)},\\
		3/4-l<0\Rightarrow 2^{k\left(s-2-l\right)},\\
		3/4-l=0\Rightarrow k2^{k\left(s-2-l\right)}.
	\end{array}\right.
\end{equation}
By summation over $k$, $s<3/4$ is fine if $l \geq -1/4, s \leq l+2$ and $(s, l) \neq \left(7/4, -1/4\right)$; $s>3/4$ is fine if $s \leq \min\{l+2, 2l+5/4\}$ and $(s, l) \neq \left(11/4, 3/4\right)$; and $s=3/4$ is fine if  $l \geq -1/4$ and $(s, l) \neq \left(3/4, -1/4\right)$.

Ultimately, an application of  H\"{o}lder's inequality for $t$, \textbf{Case 1}, \textbf{Case 2} and \textbf{Case 3} together with $\left\|P_{k} M\right\|_{H^{l}} \in \ell_{k}^{2}$ lead to $\eqref{boundary2}$.

The proofs for (2) and (3) are similarly obtained.

	\end{proof}
\begin{remark}
	
	(1) Concerning $\textbf{Case 1}$ in the proof for \eqref{MNu}, we explain how to choose $p, p_{i}, q_{i}$. The choices, all based on the confined conditions in $\eqref{crucial condition2}$, are to obtain a wider range of $(s, l)$.
	
	Step 1: We first determine $p$. In terms of  \eqref{case11}, it is optimal the differences $\frac{1}{q_1}-\frac{1}{p_1} \geq 0$, $\frac{1}{q_2}-\frac{1}{p_2} \geq 0$ and $\frac{1}{q_3}-\frac{1}{p_3} \geq 0$ are all as small as possible, with 0 being the best. On one hand, $p_{1}$, $p_2$ and $p_{3}$ should be as small as possible, which means $p$ should be small with $\frac{1}{p}=\frac{1}{p_1}+\frac{1}{p_2}+\frac{1}{p_3}$. In addition,  as $p$ comes from the Schr\"{o}dinger Strichartz estimates in Lemma $\ref{strichartz}$, together with the space $\eqref{Xs}$,  it is optimal to choose $p=\frac{6q}{3q+4}$. 
	
	On the other hand, $q_{1}$, $q_{2}$ and $q_{3}$ should all be as large as possible. Thus, in view of  $\eqref{Xs}$ and $\eqref{Yl}$, we observe $q_{1}=2$, $q_{2}=\frac{6q}{3q-4}$ and $q_{3}=2$. 
	
	Step 2: Since $k$ is dominant in the summation of the first dyadic decomposition, we have to choose $p_{1}$ preferentially. The condition that $p_{1}$ should be as small as possible, together with $p_{1} \geq \max \{q_{1}, p\}$ is briefly described as $p_{1} = \max \{q_{1}, p\}$; that is to say $p_{1}=2$. At this time, $\frac{1}{p}=\frac{1}{p_1}+\frac{1}{\sigma}$ derives $\sigma=\frac{3q}{2}$.
	
	Step 3: Similar to the above Step 2, $k_{2}$ is dominant in the summation of the second dyadic decomposition. Hence $p_{2} = \max \{q_{2}, \sigma\}=\frac{6q}{3q-4}$, where $\frac{6q}{3q-4} \geq \frac{3q}{2}$ deduces that the range of  $q$ is $2 < q \leq \frac{8}{3}$. 
	
	Step 4:  Now $p_{3}=\frac{6q}{8-3q}$ by $\frac{1}{\sigma}=\frac{1}{p_2}+\frac{1}{p_3}$, and we have $p_{3} \geq \max\{q_{3}, \sigma\}$.
	
	Step 5: This is a test process. Steps 1 to 4 are just related to the spatial variable in theory. However, considering the spaces $\eqref{Xs}$, $\eqref{Yl}$ and $2<q\leq \frac{8}{3}$, it is imperative to check whether the norms of $u, N, M$ for the time variable $t$ are still reasonable when using H\"{o}lder's inequalities. If not, adjustments must be made for $p, p_{i}, q_{i}$. For instance, the parameter selection for proving \eqref{uvw} is like this, leaving it to the reader.
	
	Thus, we choose $ (p, \sigma, p_1, p_2, p_3, q_1, q_2, q_3) = \left(\frac{6q}{3q+4}, \frac{3q}{2}, 2,  \frac{6q}{3q-4}, \frac{6q}{8-3q}, 2, \frac{6q}{3q-4}, 2\right)$ in \textbf{Case 1}. 
	
	(2) Differently to \textbf{Case 1} and \textbf{Case 2}, $k_{3}$ is dominant in the summation of the second dyadic decomposition in \textbf{Case 3}. That is why we will first choose $p_{3}= \max \{q_{3}, \sigma\}=\frac{3q}{2}$ for $2 < q \leq \frac{8}{3}$ in Step 3. 
	\end{remark}
\subsection{Proof of Theorem $\ref{Theorem2}$}
It is time to prove Theorem $\ref{Theorem2}$ by applying the multilinear estimates above. For any initial data $\left(u_{0}, N_{0}\right) \in H^{s}\left(\mathbb{R}^{3}\right) \times H^{l}\left(\mathbb{R}^{3}\right)$, we define a mapping of $(u, N) \mapsto \Phi_{u_{0}, N_{0}}(u, N)=(F(u,N),G(u,N))$, where $F(u,N)$ and $G(u,N)$ are the right-hand sides of equations $\eqref{u}$ and $\eqref{N}$, respectively. 
Now, the proof proceeds by a standard contraction mapping principle. First, the intersection of ranges for $q$ yields that $q=8/3$ is the optimal choice. In addition, the resolution space is denoted as
\begin{equation}
		X(T) =\left\{(u,N): \left\|(u,N) \right\|_{X(T)}\leq \eta=5/2 \left(\left\|u_0\right\|_{H^s}+\left\|N_0\right\|_{H^l}\right) \right\},
\end{equation} 
where the $X(T)$-norm is given by
$$\left\|(u,N) \right\|_{X(T)} =\left\|u\right\|_{X^s} + \left\|N\right\|_{Y^l}.$$ 
By the previous lemmas, we have 
\begin{equation}
	\begin{aligned}
		&\left\|F(u,N)\right\|_{X^s} \\
		&\lesssim\left\|u_0\right\|_{H^s} +2^{-\beta \theta_{1}}\left\|N_{0}\right\|_{H^l}\left\|u_0\right\|_{H^s} +2^{-\beta (\theta_{1}+\theta_{2})} \left\|N\right\|_{Y^l} \left\|u\right\|_{X^s}\\
		&\quad+ T^{1/4}(C(\beta)+1)\left\|N\right\|_{Y^l} \left\|u\right\|_{X^s} +T^{1/4}\left\|u\right\|^{3}_{X^{s}}+T^{1/4}\left\|N\right\|^2_{Y^l}\left\|u\right\|_{X^s} ,
	\end{aligned}
\end{equation} 
and
\begin{equation}
	\begin{aligned}
		\left\|G(u,N)\right\|_{Y^l}&\lesssim \left\|N_0\right\|_{H^l} +2^{-\beta \theta_{3}}\left\|u_{0}\right\|^2 _{H^s} + 2^{-\beta \theta_{3}}\left\|u\right\|^2_{X^s} \\
		&\quad+T^{3/8}(C(\beta)+1) \left\|u\right\|^2_{X^s}+T^{1/4} \left\|N\right\|_{Y^l} \left\|u\right\|^2_{X^s}.
	\end{aligned}
\end{equation} 
Without loss of generality, we denote $\theta=\min\left\{\theta_{1}, \theta_{3}\right\}$, and  $T^{\delta}=\max \left\{T^{1/4}, T^{3/8}\right\}$. Hence for any $(u, N), (v, M) \in X(T)$ with the same initial data, it follows that
\begin{equation}
	\begin{aligned}
		&\|\Phi_{u_{0}, N_{0}}(u, N)-\Phi_{v_{0}, M_{0}}(v, M)\|_{X(T)}\\
		&\leq \left(2^{-\beta \theta}\eta
		+T^{\delta}(C(\beta)+1)\eta+T^{\delta}C\eta^{2}\right)\left(\|u-v\|_{X^{s}}+\|N-M\|_{Y^{l}}\right).
	\end{aligned}
\end{equation}
We let $\beta=\beta(\eta)$ satisfy $2^{-\beta \theta}\eta \leq 1/4$ and choose an appropriate $T>0$ such that $T^{\delta}(C(\beta)+1)\eta+T^{\delta}C\eta^{2} \leq 1/4$. Hence we conclude that
\begin{equation}
	\|\Phi_{u_{0}, N_{0}}(u, N)-\Phi_{v_{0}, M_{0}}(v, M)\|_{X(T)} \leq 1/2 \left(\|u-v\|_{X^{s}}+\|N-M\|_{Y^{l}}\right).
\end{equation}
In addition,  by denoting $(v,M) = (0,0)$ we get
\begin{equation}
	\begin{aligned}
	\left\|\Phi_{u_{0}, N_{0}}(u, N) \right\|_{X(T)}
	\leq &\eta/2+\left\|u_{0}\right\|_{H^{s}}+\left\|N_{0}\right\|_{H^{l}}+2^{-\beta \theta}\left\|N_{0}\right\|_{H^{l}}\left\|u_{0}\right\|_{H^{s}}\\
	&+2^{-\beta \theta}\left\|u_{0}\right\|^{2}_{H^{s}}\leq \eta.
	\end{aligned}
\end{equation}
Through this, we can see that $\Phi$ is a contraction mapping for any initial data $(u_{0}, N_{0})$. Thus, by the contraction mapping principle , we have a local unique solution in $X^{s} \times Y^{l},$ and the Lipschitz continuity  $H^{s} \times H^{l} \rightarrow X^{s} \times Y^{l}$ follows from the standard argument.

\section{Local well-posedness for $d=2$}
This section is devoted to the case for dimension 2. The spaces here are defined as
\begin{equation}\label{Xs'}
	\begin{aligned}
	u \in X^{s}:= &C \left([-T, T];H^{s}\left(\mathbb{R}^{2}\right)  \right) \cap L^{\infty}\left([-T,T];H^{s}\left(\mathbb{R}^{2}\right) \right)\\ 
	&\cap L^{q}\left([-T,T];B^{s}_{\frac{2q}{q - 2}}\left(\mathbb{R}^{2}\right) \right)
	\end{aligned}
\end{equation}
for $2< q \leq \infty$, and
\begin{equation}\label{Yl'}
	N \in Y^{l}:= C \left([-T,T];H^{l}\left(\mathbb{R}^{2}\right) \right) \cap L^{\infty}\left([-T,T];H^{l}\left(\mathbb{R}^{2}\right) \right).
\end{equation}
We ultimately determine $q=4$ for $d=2$. Since the proofs are essentially the same as in the case for dimension 3, we only present the main lemmas.
\begin{lemma} (Quadratic terms)\label{Quadratic terms d=2} 
	
	(1) If $s, l \geq 0$, then for any $N(x)$ and $u(x)$,
	\begin{equation}\label{Quadratic 1d=2}
		\left\|(N u)_{L H+HH}\right\|_{L_{t}^{4/3}B_{4/3}^{s}} \lesssim T^{1/2}\|N\|_{L_{t}^{\infty}H^{l}}\|u\|_{L_{t}^{4}B_{4}^{s}}.
	\end{equation}
	\begin{equation}\label{Quadratic 11d=2}
		\left\|(N u)_{\alpha L}\right\|_{L_{t}^{4/3}B_{4/3}^{s}} \lesssim T^{1/2}C(\beta)\|N\|_{L_{t}^{\infty}H^{l}}\|u\|_{L_{t}^{4}B_{4}^{s}}.
	\end{equation}
	
	(2) If $2s \geq l+1$, then for any $u(x)$ and $v(x)$,
	\begin{equation}\label{Quadratic 2d=2}
		\left\|D(u v)_{H H }\right\|_{L_{t}^{8/7}B^l_{4/3}} \lesssim T^{5/8}\|u\|_{L^{\infty}_{t}H^{s}}\|v\|_{L^{4}_{t}B_{4}^{s}}.
	\end{equation}
	\begin{equation}\label{Quadratic 22d=2}
		\left\|D(u v)_{\alpha L+L \alpha }\right\|_{L_{t}^{8/7}B^l_{4/3}} \lesssim T^{5/8}C(\beta)\|u\|_{L^{\infty}_{t}H^{s}}\|v\|_{L^{4}_{t}B_{4}^{s}}.
	\end{equation}
\end{lemma}

\begin{lemma} (Boundary terms)\label{Boundary terms d=2} For any $\theta_{i}(s, l) \geq 0$, and functions $N(x), u(x),  v(x)$, we have the following:
	
	(1) If $l \geq -1$, $s \leq l+2$ and $(s, l) \neq \left(1, -1\right)$, then
	\begin{equation}\label{boundary1d=2}
		\|\Omega(N, u)\|_{L_{t}^{\infty}H^{s}} \lesssim 2^{- \beta \theta_{1}}\|N\|_{L_{t}^{\infty}H^{l}}\|u\|_{L_{t}^{\infty}H^{s}}.
	\end{equation}
	
	(2) If $l \geq -1$, $s \leq l+3/2$ and $(s, l) \neq \left(1/2, -1\right)$, then
	\begin{equation}\label{boundary2d=2}
		\|\Omega(N, u)\|_{L_{t}^{4}B^{s}_{4}} \lesssim 2^{- \beta \theta_{2} }\|N\|_{L_{t}^{\infty}H^{l}}\|u\|_{L_{t}^{4}B^{s}_{4}}.
	\end{equation}
	
	(3) If $s \geq \text{max}\left(l-1, l/2\right)$ and $(s, l)\neq (1, 2)$, then
	\begin{equation}\label{boundary3d=2}
		\|D \tilde{\Omega}(u, v)\|_{L_{t}^{\infty}H^{l}} \lesssim 2^{- \beta \theta_{3}}\|u\|_{L_{t}^{\infty}H^{s}}\|v\|_{L_{t}^{\infty}H^{s}}.
	\end{equation}
\end{lemma}

\begin{lemma} (Cubic terms)\label{Cubic terms d=2}  For any $M(x), N(x), u(x), v(x)$, we have the following:
	
	(1) If  $l \geq -1/2, s \leq \min\{l+2, 2l+3/2\}$ and $ (s, l) \neq \left(3/2, -1/2 \right), \left(5/2, 1/2\right),\\ \left(1/2, -1/2\right)$, then
	\begin{equation}\label{MNu2}
		\|\Omega(M, N u)\|_{L^{4/3}_{t}B_{4/3}^{s}} \lesssim T^{1/2}\|M\|_{L^{\infty}_{t}H^{l}}\|N\|_{L^{\infty}_{t}H^{l}}\|u\|_{L^{4}_{t}B_{4}^{s}}.
	\end{equation}
	
	(2) If $s \geq -1/4$, then
	\begin{equation}
		\|\Omega(D(u v), w)\|_{L^{4/3}_{t}B_{4/3}^{s}} \lesssim T^{1/4} \|u\|_{{L_{t}^{4}}B_{4}^{s}}\|v\|_{{L_{t}^{4}}B_{4}^{s}}\|w\|_{L_{t}^{\infty}H^{s}}.
	\end{equation}
	
	(3) If $s \geq 0$, $l \leq \min\{s+1, 2s+1/2\}$ and $(s ,l) \neq \left(1/2, 3/2\right), \left(0, 1/2\right)$, 
	then 
	\begin{equation}
		\|D \tilde{\Omega}(N u, v)\|_{L^{1}_{t}H^{l}}+\|D \tilde{\Omega}(v, N u)\|_{L^{1}_{t}H^{l}} \lesssim T^{1/2}\|N\|_{L^{\infty}_{t}H^{l}}\|u\|_{L^{4}_{t}B_{4}^{s}}\|v\|_{L^{4}_{t}B_{4}^{s}}.
	\end{equation}
\end{lemma}

\section*{Acknowledgments} We would like to thank Prof. Zihua Guo for the numerous discussions and encouragement. We are also grateful to the anonymous referees for the helpful comments which greatly improve the presentation of this paper. The second author is partially supported by the Chinese Scholarship Council (No. 201906050022).

\end{document}